\documentclass[12pt]{article}
\usepackage{xcolor,graphicx} 
\usepackage{amsmath,amsthm,amssymb}

\usepackage{ifthen}

\usepackage{amsmath,amsthm,amssymb,mathtools}
\usepackage{pgf}
\usepackage{graphicx}
\usepackage{tikz}
\usetikzlibrary{backgrounds, patterns, shapes, calc, decorations.pathreplacing, decorations.pathmorphing}

\tikzset{%
    vtx/.style={draw,circle,very thin,white,fill=black,inner sep=0pt,text width=8pt},
    pth/.style={draw,decorate,decoration={snake},thick},
    edg/.style={ultra thick},
    smedg/.style={very thin},
    polyg/.style={very thin,fill=white},
    dotedg/.style={ultra thick,dotted},
    ellip/.style={very thin,black,fill=white},
    spedg/.style={rounded corners,line width=5pt,purple!50}, 
}
\tikzset{/pgf/foreach/parse=true}
\pgfdeclarelayer{backback}
\pgfdeclarelayer{back}
\pgfdeclarelayer{fore}
\pgfdeclarelayer{forefore}
\pgfsetlayers{backback,back,main,fore,forefore}

\newtheorem{theorem}{Theorem}[section]

\newtheorem{obs}[theorem]{Observation}
\newtheorem{prop}[theorem]{Proposition}

\newtheorem{lemma}[theorem]{Lemma}
\newtheorem{proposition}[theorem]{Proposition}

\newtheorem{claim}[theorem]{Claim}

\newtheorem{remark}[theorem]{Remark}

\newtheorem{conjecture}[theorem]{Conjecture}

\def\beq{\begin{equation}}\def\eeq{\end{equation}}
\def\beqn{\begin{eqnarray}}\def\eeqn{\end{eqnarray}}

\def\qed{\ifhmode\unskip\nobreak\fi\quad\ifmmode\Box\else$\Box$\fi}

\begin{document}
\title{Bounded diameter variations of Ryser's conjecture}
\author{Andr\'as Gy\'arf\'as\thanks{Alfr\'ed R\'enyi Institute of Mathematics, Budapest, P.O. Box 127, Budapest, Hungary, H-1364.
\texttt{gyarfas.andras@renyi.hu}, \texttt{sarkozy.gabor@renyi.hu}} \thanks{Research supported in part by
NKFIH Grant No. K132696.} \and G\'{a}bor N. S\'ark\"ozy\footnotemark[1]
\thanks{Computer Science Department, Worcester Polytechnic Institute, Worcester, MA.} \thanks{Research supported in part by
NKFIH Grants No. K132696, K117879.}}


\maketitle
\begin{abstract}
In this paper we study bounded diameter variations of the following form of Ryser's conjecture. For every graph $G=(V,E)$ with independence number $\alpha(G)=\alpha$ and integer $r\geq 2$, in every $r$-edge coloring of $G$ there is a cover of $V(G)$ by the vertices of $(r-1)\alpha$ monochromatic connected components. Mili\'{c}evi\'{c} initiated the question whether the diameters of the covering components can be bounded.

For any graph $G$ with $\alpha(G)=2$  we show that in every 2-coloring of the edges, $V(G)$ can be covered by the vertices of two monochromatic subgraphs of diameter at most 4.
This improves a result of DeBiasio et al., which in turn improved a result of Mili\'{c}evi\'{c}. It remains open whether diameter $4$ can be strengthened to diameter $3$, we could do this only for certain graphs, including odd antiholes.

We propose also a somewhat orthogonal aspect of the problem. Suppose that we fix the diameter $d$ of the
monochromatic components, how many do we need to cover the vertex set? For $d=2,2\le r \le 3$, the exact answer is $r\alpha$ and for $d=4,r=2$, we prove the upper bound $\lfloor 3\alpha/2\rfloor$.

\end{abstract}

\section{Introduction}

If $G$ is a graph whose edges are colored with two colors, $G_1, G_2$ denote the subgraphs defined by the edges of colors 1,2, respectively.  For convenience we always assume that $G_i, G_j$ refer to different indices in $\{1,2\}$. When it is more convenient, we refer to colors 1, 2 as red and blue, respectively. $G^c$ denotes the {\em complement} of a graph $G$. The graph $H=(V,F)$ is a {\em spanning subgraph} of the graph $G=(V,E)$ if $F\subset E$ and every vertex of $G$ is incident to some edge of $F$. The neighborhood of vertex $v$ is denoted by $N(v)$, and by $N_i(v)$ in color $i$.
If we connect the centers of two stars, we get a {\em double star}, the connecting edge is called the {\em base}. If we connect the centers of three stars with a path of length 2, we get a {\em triple star}.
$K_n$ denotes the complete graph on $n$ vertices.

{\em Perfect graphs} are the graphs $G$ in which all induced subgraphs $H\subseteq G$ satisfy the property $\omega(H)=\chi(H)$, where $\omega(H),\chi(H)$ are the size of the largest complete subgraph and the chromatic number.  Berge had two conjectures, the weaker one was that complements of perfect graphs are also perfect, i.e. all induced subgraphs $H\subseteq G$ satisfy the property $\alpha(H)=\theta(H)$, where $\alpha(H),\theta(H)$ are the size of the largest independent set and the clique-cover number (the minimum number of complete subgraphs whose vertices cover the whole vertex set). This was proved by Lov\'asz \cite{LO2} (weak perfect graph theorem). The stronger conjecture stated that perfect graphs can be characterized by excluding odd cycles of length at least 5 (odd holes) and their complements (odd antiholes) as induced subgraphs. The stronger conjecture was proved by Chudnovsky, Robertson, Seymour and Thomas \cite{CRST} (strong perfect graph theorem).

In a graph $G=(V,E)$ and for vertices $u, v\in V(G)$, let $d_G(u,v)$ (or just $d(u,v)$) denote the {\em distance} of $u$ and $v$ in $G$, i.e. the length of the shortest path connecting $u$ and $v$ in $G$. If there is no path connecting $u$ and $v$, let $d(u,v)=\infty$. The {\em diameter} of a graph $G$, denoted by $diam(G)$, is the smallest integer $d$ for which $d(u,v)\leq d$ for all $u, v \in V(G)$. The diameter of a one-vertex graph is defined as $0$.  If we have a $(G_1, G_2)$ edge-coloring of a graph $G$, sometimes we call $diam(G_1)$ the {\em red diameter}, $diam(G_2)$ the {\em blue diameter} of $G$.

For a graph $G$, a {\em substitution} into $G$ is a graph obtained by replacing each vertex $v$ of $G$ with an arbitrary graph $G_v$ and replacing each $(u,v)\in E(G)$ with a complete bipartite graph between $G_u$ and $G_v$. Note that if $G'$ is obtained by a substitution into $G$ with $diam(G) \ge 2$ then $diam(G')=diam(G)$.

The following is a reformulation (from \cite{GYsztaki}, more accessible in \cite{DEB}, page 3) of a classical conjecture attributed to Ryser (stated in the thesis of Henderson \cite{HEN}, a more general form is due to Lov\'asz \cite{LO}).

\begin{conjecture}
For every graph $G=(V,E)$ with $\alpha(G)=\alpha$ and integer $r\geq 2$, in every $r$-edge coloring of $G$ there is a cover of $V(G)$ by $(r-1)\alpha$
monochromatic connected components.
\end{conjecture}

There has been extensive research on this conjecture, see the excellent survey by DeBiasio et al. \cite{DEB} on various partial results, generalizations and strengthenings
of the conjecture. For $r = 2$, this is just K\"onig’s theorem \cite{KO}, and the $r = 3$ case was proved
by Aharoni \cite{AH}. While Ryser’s conjecture is still wide open for all $r \geq 4$, for $\alpha(G) = 1$ (i.e. for complete graphs) the conjecture holds for $r=4$ (\cite{GYsztaki}, \cite{TU1}) and for $r=5$ (\cite{TU2}).

Recently a new direction emerged starting with the work of Mili\'{c}evi\'{c} \cite{MI1,MI2}, where bounded diameter
variations of the conjecture were studied. In \cite{DEB} it was conjectured that we can guarantee in Ryser's
conjecture that the diameter of the monochromatic connected components is at most $f(r, \alpha)$ for some function $f$. In \cite{DGHS} it was proved
that $f(2, \alpha) \leq 8\alpha^2+12\alpha+4$. In \cite{DEB} the authors speculated that perhaps it is possible that $f$ is a constant; it is independent of both $r$ and $\alpha$.
This fascinating question is our starting point here in this paper.

The diameter problem for the  $r=2, \alpha=1$ case is answered by a well-known (folklore) property of $2$-colored complete graphs. It can be found in \cite{BO}, \cite{DGHS}, \cite{EMMP}, \cite{GY}, \cite{W}.
\begin{proposition}\label{simplefolk} Every 2-colored complete graph has a spanning monochromatic subgraph of diameter at most 3.
\end{proposition}

It is worth mentioning that Proposition \ref{simplefolk} has the following consequence.

\begin{proposition}\label{perfectgraphs} The vertex set of every 2-colored perfect graph $G$ can be covered by  the vertices of $\alpha(G)$ monochromatic subgraphs, each with diameter at most $3$.
\end{proposition}

Indeed, $V(G)$ can be covered by the vertices of  $\alpha(G)$ complete subgraphs of $G$ so Proposition \ref{simplefolk} implies Proposition \ref{perfectgraphs}.

\section{Results}

\subsection{Graphs with $\alpha(G)=2$}

We say that a $2$-colored graph $G$ has a $(p,q)$-cover if $V(G)$ can be covered by the vertices of two monochromatic subgraphs, one with diameter at most $p$, the other with diameter at most $q$. Our main result is the following.

\begin{theorem}\label{main}  Assume that $G$ is a 2-colored graph with $\alpha(G)=2$. We have $f(2,2)\leq 4$. i.e. $V(G)$ has a $(4,4)$-cover.
\end{theorem}

A weaker form of Theorem \ref{main}, a $(6,6)$-cover was proved in \cite{DEB} which in turn improved the $(8,8)$-cover from \cite{MI1}.
We were not able to determine whether we have always a $(3,3)$-cover, but a $(2,2)$-cover does not necessarily exist (see Remark \ref{const} below).

For some graphs there are certainly $(3,3)$-covers, for example for graphs with a bipartite complement; this follows from Proposition \ref{simplefolk}.  Since a graph $G$ with $\alpha(G)=2$ and with a non-bipartite complement contains an odd antihole (see Proposition \ref{bipcomp} below), it is natural to consider the case when $G$ is an odd antihole. We show that odd antiholes also have $(3,3)$-coverings. In fact, we prove the following stronger result.


\begin{theorem}\label{antiby3} Assume that $G$ is a $2$-colored graph with $\alpha(G)=2$ such that for some $v\in V(G)$,  $V(G)\setminus \{v\}$ is the union of two vertex disjoint complete graphs $K'$ and $K''$ and $v$ is adjacent to all but one vertex of both $K'$ and $K''$.  Then $G$ has a $(3,3)$-cover.
\end{theorem}

Note that a $(2,2)$-cover cannot be guaranteed in Theorem \ref{antiby3}. Indeed, if $G$ is an antihole with seven vertices then the edges in $G$ can be partitioned into a red and a blue Hamiltonian cycle. Since any monochromatic subgraph of diameter 2 has at most three vertices, there is no way to cover $V(G)$ by two of them.

\subsection{Coverings by components with a fixed diameter}\label{orth}

The conjectures and results about bounding the diameters of the covering components can be considered  from a somewhat orthogonal aspect as well. Suppose that we fix the diameter $d$ of the
monochromatic connected components, how many do we need to cover the vertex set?
Let $g(r,d,\alpha)$ be the minimum $k$ for which every $r$-colored graph $G$ with $\alpha(G)=\alpha$, $V(G)$ has a cover by the vertices of $k$ monochromatic subgraphs of diameter at most $d$.
For $d=2$, the following is an easy observation.
\begin{obs}\label{d=2}  Assume that $G$ is an r-colored graph with $\alpha(G)=\alpha$. We have $g(r,2,\alpha)\leq r \alpha$. i.e. $V(G)$ can be covered by the vertices of $r\alpha$ monochromatic connected components, each with diameter at most 2.
\end{obs}
Indeed, in a graph with $\alpha(G)=\alpha$ the set of edges incident to an independent set of size $\alpha$ form a spanning subgraph, thus at most $r\alpha$ monochromatic stars give the required cover.

\begin{remark}\label{const} Observation \ref{d=2} is sharp for $r=2,3$ in the sense that $r\alpha-1$ components are not always enough for a covering, so $g(r,2,\alpha)=r \alpha$.
\end{remark}
Indeed, for $r=2$ one can take $\alpha$ vertex disjoint copies of a $2$-colored $K_4$ where both colors form a path with three edges (let us denote this special 2-colored $K_4$ by $P_4^2$, see Figure 1, this will play a role later). For $r=3$ take $\alpha$ copies of a $3$-colored $K_7$ where all colors form a Hamiltonian cycle.  Since the largest monochromatic diameter 2 subgraphs in these colorings have three vertices, each of the $\alpha$ parts needs at least two (respectively three) components for the required cover.

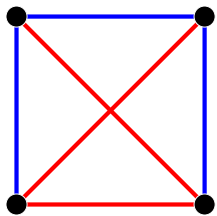
\begin{figure}
		       \centering
                \begin{tikzpicture}[scale=2.5]
                    \centering
				\useasboundingbox (0,-0.2) rectangle (1,1);
				\coordinate (a) at (0,0);
				\coordinate (b) at (1,0);		
				\coordinate (c) at (0,1);
				\coordinate (d) at (1,1);
					\begin{pgfonlayer}{forefore}
					\draw (a) node[vtx] {};
					\draw (b) node[vtx] {};
					\draw (c) node[vtx] {};
					\draw (d) node[vtx] {};
					\end{pgfonlayer}
				\begin{pgfonlayer}{fore}
					\draw[edg,red] (a) -- (b);
					\draw[edg,red] (b) -- (c);
					\draw[edg,blue] (c) -- (d);
					\draw[edg,red] (d) -- (a);
					\draw[edg,blue] (d) -- (b);
					\draw[edg,blue] (a) -- (c);
					\end{pgfonlayer}
			\end{tikzpicture}
			\caption{The special 2-colored $K_4$ denoted by $P_4^2$}
            \label{P_4^2}
	\end{figure}

For $d=4,r=2$ and a general $\alpha$ we have the following.

\begin{theorem}\label{3/2}  Assume that $G$ is a 2-colored graph with $\alpha(G)=\alpha$. We have $g(2,4,\alpha)\leq \lfloor 3\alpha/2 \rfloor$. i.e. $V(G)$ can be covered by the
vertices of $\lfloor 3\alpha/2 \rfloor$ monochromatic connected components, each with diameter at most 4.
\end{theorem}


\section{Tools}\label{tools}

We will use the following easy observation several times.

\begin{proposition}\label{szor}
Let $G$ be a graph with $diam(G)=2$.
\begin{itemize}
\item[(i)] Let us add a star to $G$ from an arbitrary vertex of $G$ to new vertices. Denote the resulting graph
by $G'$, then $diam(G')\leq 3$.
\item[(ii)] Let us add several disjoint stars to $G$ from vertices of $G$ to new vertices. Denote the resulting graph
by $G''$, then $diam(G'')\leq 4$.
\end{itemize}
\end{proposition}

The following is also a simple but useful proposition.

\begin{prop}\label{bipcomp} If $G^c$ is bipartite, i.e. $G$ is the union of two disjoint cliques, then $G$ has a $(3,3)$-cover. If $G^c$ is not bipartite and $\alpha(G)=2$ then $G$ contains an odd antihole.
\end{prop}

\begin{proof}The first statement follows from Proposition \ref{simplefolk} above. For the second statement, observe that an odd cycle of minimum length in $G^c$ must be an odd antihole.
\end{proof}

We will also need a slight strengthening of Proposition \ref{simplefolk}. It classifies $2$-colored complete graphs into four mutually exclusive categories.
To this end we define two families of special $2$-colored complete graphs, ${\cal{H}}_1$ and ${\cal{H}}_2$.

Let the two colors be colors 1 and 2. Let $\{x_1,x_2,x_3,x_4,x_5\}$ be the vertices of the complete graph $K_5$.
The edges of the four-cycle $(x_1,x_2,x_3,x_4)$ and the edges of the triangle $(x_1,x_2,x_5)$ are colored with color 1 (forming  a ``house'' graph).  The edges $(x_1,x_3)$ and $(x_2,x_4)$ are colored with color 2. The color of the edges $(x_4,x_5)$ and ($x_3,x_5)$ are not specified (see Figure 2 where color 1 is red). Then we define ${\cal{H}}_1$ as the family obtained from this special 2-colored $K_5$ by the following substitutions. Replace vertex $x_i$  with a vertex-set $A(i)$, $|A(i)|=n_i$, obeying the following requirements:  $n_1=n_2=1, n_3,n_4, n_5\ge 0$ (i.e. $x_3, x_4, x_5$ might be deleted). Substitute arbitrary $2$-colored complete graphs into $A(i)$.  For $i\ne j$, the  edges of  $[A_i,A_j]$ are colored with the color of $(x_i,x_j)$ if the color is specified, otherwise (in $[A_3,A_5],[A_4,A_5]$) the coloring is arbitrary. We get ${\cal{H}}_2$ by swapping the two colors.
Note that any graph in ${\cal{H}}_1$ has diameter at most $2$ in color 1, while any graph in ${\cal{H}}_2$ has diameter at most $2$ in color 2.
Furthermore, in any graph in ${\cal{H}}_1$, the edge $(x_1,x_2)$ and the edges of the bipartite graphs $[x_1,A_3],[x_2,A_4]$ are base edges of spanning double stars in color 1 (similarly in ${\cal{H}}_2$).

\begin{figure}
		       \centering
                \begin{tikzpicture}[scale=2.5]
                    \centering
				\useasboundingbox (0,-0.2) rectangle (1,1.5);
				\coordinate (a) at (0,0);
				\coordinate (b) at (1,0);		
				\coordinate (c) at (0,1);
				\coordinate (d) at (1,1);
				\coordinate (e) at (0.5,1.5);
					\begin{pgfonlayer}{forefore}
					\draw (a) node[vtx,label={left:$x_4$}] {};
					\draw (b) node[vtx,label={right:$x_3$}] {};
					\draw (c) node[vtx,label={left:$x_1$}] {};
					\draw (d) node[vtx,label={right:$x_2$}] {};
                    \draw (e) node[vtx,label={above:$x_5$}] {};
					\end{pgfonlayer}
				\begin{pgfonlayer}{fore}
					\draw[edg,red] (a) -- (b);
					\draw[edg,red] (c) -- (d);
					\draw[edg,red] (c) -- (a);
					\draw[edg,red] (b) -- (d);
					\draw[edg,blue] (c) -- (b);
					\draw[edg,blue] (a) -- (d);
					\draw[edg,red] (c) -- (e);
					\draw[edg,red] (d) -- (e);
					\end{pgfonlayer}
			\end{tikzpicture}
			\caption{The ``skeleton" of the graphs in ${\cal{H}}_1$}
            \label{P_4^2}
	\end{figure}
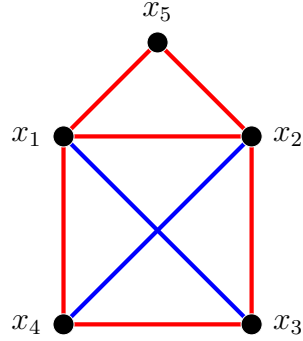

\begin{lemma}\label{folk}  Assume that $G$ is a 2-colored complete graph with at least two vertices, $G=G_1\cup G_2$, $diam(G_1)\ge diam(G_2)$.  Then exactly one of the following possibilities holds.
\begin{itemize}
\item[(i)]  $diam(G_1)>3$. Then we must have $G \in {\cal{H}}_2$ and $diam(G_2)\leq 2$.
\item[(ii)] $diam(G_1)=diam(G_2)=3$. Then $G_1,G_2$ both have spanning double stars and their base edges span a $P_4^2$.
\item[(iii)] $diam(G_1)=3, diam(G_2)=2$. Then $G_2$ has a spanning double star.
\item[(iv)] $diam(G_1)=diam(G_2)=2$.
\end{itemize}
\end{lemma}
Note that in cases (i),(ii),(iii) we have specific subgraphs witnessing that $G_2$ (or both $G_1$ and $G_2$) have a small diameter. In case (iv) no specific witness exists since in random 2-colorings $diam(G_1)=diam(G_2)=2$.
\begin{proof}

\noindent
If $diam(G_1)>3$, then there must exist two vertices $x_1,x_2$ such that  $d_{G_1}(x_1,x_2)>3$. This implies that $(x_1,x_2)\in E(G_2)$, $N_1(x_1)\cap N_1(x_2)=\emptyset$
and that there is no edge of color 1 between $N_1(x_1)$ and $N_1(x_2)$.
Thus all of the following edges must be of color 2: edges from $x_1$ to $N_1(x_2)$, from $x_2$ to $N_1(x_1)$, between $N_1(x_1)$ and $N_1(x_2)$ and finally
from $\{x_1,x_2\}$ to $S=V(G)\setminus (\{x_1,x_2\}\cup N_1(x_1)\cup N_1(x_2))$.
This implies $G\in {\cal{H}}_2$.
Indeed, in addition to $\{x_1, x_2\}$, the sets $N_1(x_1), N_1(x_2), S$, are obtained by substitutions into $x_3, x_4,x_5$, respectively. Clearly, $diam(G_2)\le 2$, proving (i).
\smallskip

If $diam(G_1)=diam(G_2)=3$, then we must have two vertices $x_1,x_2$ such that $d_{G_1}(x_1,x_2)=3$. This implies that $(x_1,x_2)\in E(G_2)$ and $N_1(x_1)\cap N_1(x_2)=\emptyset$.
Thus indeed the edge $e=(x_1,x_2)$ is the base edge of a spanning double star in color $2$. We can apply this argument to $G_2$ as well
to get a double star in color 1 with base edge $f=(x_3, x_4)$. Observe first that $e,f$ must be vertex disjoint edges because of the definition of the spanning double stars.
Furthermore, the edge $e$ sends exactly one edge of color 2 to $x_3$ and exactly one edge of color 2 to $x_4$ and similarly for the edge $f$ (using the fact that these are spanning
double stars). Then indeed $e$ and $f$ span a $P_4^2$, i.e. both colors form a path of length 3, proving (ii).
\smallskip

The remaining two possibilities  are (iii) and (iv) (note that if $diam(G_2)=1$, then $diam(G_1)=\infty$ and this case is covered in (i)).
\end{proof}

It is worth mentioning that (ii) in Lemma \ref{folk} can be strengthened as follows (this is Lemma 1 in \cite{GYfruit}).

\begin{remark}   Let $B_1,B_2$ denote the set of edges in $G$ that are base edges of some spanning double star in $G_1,G_2$, respectively. Then $B_1,B_2$ form vertex-disjoint bipartite graphs (and any $b_1\in B_1, b_2\in B_2$ span a $P_4^2$).
\end{remark}

\section{Proofs}




\subsection{Proof of Theorem \ref{main}}

We start with some notation to describe the structure  of a 2-colored $G$ with $\alpha(G)=2$ with respect to a pair of nonadjacent vertices $x, y$. (Note that we are using a similar notation and outline
as in the proof for diameter 6, Theorem 43 in \cite{DEB}.) The colors are denoted by $1,2$ but sometimes it is more convenient to use red for color 1 and blue to color 2.  With respect to $\{x,y\}$ the other vertices can be partitioned into eight (possibly empty) parts because no vertex in $V(G)\setminus \{x,y\}$ can be non-adjacent to both $x$ and $y$
(that would create an independent set of size 3). The parts containing vertices adjacent to both $x$ and $y$ in the same color are called {\em homogeneous} parts, and they are denoted by $A_{11}$ (red to both $x$ and $y$) and $A_{2,2}$ (blue to both $x$ and $y$). Furthermore,

\begin{itemize}
\item parts adjacent to one of $\{x,y\}$ are denoted by $A_x^1,A_x^2,A_y^1,A_y^2$,
\item parts adjacent to both $x$ and $y$ in distinct colors are denoted by $A_{12},A_{21}$,
\item $A_x=A_x^1\cup A_x^2, A_y=A_y^1\cup A_y^2$, $K_x=A_x\cup \{x\}$, $K_y=A_y\cup \{y\}$, $A=A_{11}\cup A_{22}\cup A_{12}\cup A_{21}=N(x)\cap N(y)$.
\end{itemize}

Since $x,y$ are non-adjacent to any vertex of $K_y,K_x$, respectively, and $\alpha(G)=2$, we get the following.
\begin{obs}\label{one} $K_x,K_y$ span complete graphs in $G$.
\end{obs}

Note first that if both homogeneous parts, $A_{11}, A_{22}$, are nonempty then Theorem \ref{main} is obviously true,
because the two monochromatic connected components
$$G_1[ \{x\}\cup A_x^1 \cup (A\setminus A_{22}) \cup \{y\} \cup A_y^1] \; \mbox{and} \; G_2[ \{x\}\cup A_x^2 \cup A_{22} \cup \{y\} \cup A_y^2]$$ provide a $(4,4)$-cover of $G$.
The case when there is no homogeneous part can be eliminated by choosing a suitable pair $x,y$ considering the two possibilities in Proposition \ref{bipcomp}.
Indeed, if $G^c$ is bipartite then $G$ has a $(3,3)$-cover. Otherwise, $G$ contains an odd antihole $C$. The long diagonals of $C$ (connecting a vertex to the two vertices farthest away on $C$) form a 2-colored Hamiltonian cycle which must contain two consecutive long diagonals $(a,x),(a,y)$ of the same color. Since $x,y$ are non-adjacent in $G$, $a$ belongs to a homogeneous part corresponding to $x,y$, say $a\in A_{11}$. Thus we may assume that we have exactly one non-empty homogeneous part, say $A_{11}$.

We may assume that the red diameter of at least one of $K_x$ and $K_y$ is greater than 3, otherwise
$$G_1[K_x \cup A_{11} \cup A_{12}] \; \mbox{and} \; G_1[ K_y \cup A_{21}]$$
provide a $(4,4)$ covering of $G$.
Thus, w.l.o.g. we may assume that the red diameter of $K_y$ is greater than 3, so by Lemma \ref{folk} (i) the blue diameter of $K_y$ is at most 2. Then, by using Lemma \ref{folk} (i) for $K_x$, we have two remaining cases:
either the blue diameter of $K_x$ is at most 3, or otherwise, the red diameter of $K_x$ is at most 2.

{\bf Case 1:} $K_x$ has blue diameter at most 3 (and the blue diameter of $K_y$ is at most 2).

If every vertex in $A_{11}$ sends a blue edge to $A_x^2\cup K_y$, then let $A_{11}^x$ denote the set of vertices in $A_{11}$ that send a blue edge to $A_x^2$ and let $A_{11}^y = A_{11}\setminus A_{11}^x$. We claim that
$$G_2[K_x \cup A_{11}^x \cup A_{21}] \; \mbox{and} \; G_2[ K_y \cup A_{11}^y\cup A_{12}]$$
provide a $(4,4)$-covering of $G$ by two blue components.

Indeed, the fact that the second blue component has diameter at most 4 follows immediately from Proposition \ref{szor}(ii).
For the first blue component, observe that the blue diameter $K_x$ is at most three by assumption. We claim that extending $K_x$ with $U=A_{11}^x\cup A_{21}$ spans a graph with blue diameter at most $4$. Indeed, the distance of a vertex of $U$ and a vertex of $K_x$ is obviously at most $4$. The blue distances within $A_{11}^x,A_{21}$ and from $A_{11}^x$ to $A_{21}$ are at most $4$, $2$, $3$, respectively, using $x$ for connection.

Thus we may assume that there is a vertex $z\in A_{11}$ such that every edge from $z$ to $A_x^2\cup K_y$ is red.
Let $Z= \{ v \; | \; v\in A_x^2\cup K_y, (z,v)\not\in E(G) \}$. Then $Z$ spans a complete graph in $G$, so by Proposition \ref{simplefolk} it has diameter at most 3 in one of the colors.
This is one component of the cover while the other one is
$$G_1[\{x,y\}\cup A_x^1 \cup A \cup N_1(z)],$$
having also diameter at most 4, since it is a red  triple star with centers $x$, $y$ and $z$.

{\bf Case 2:} $K_x$ has red diameter at most 2 (and the blue diameter of $K_y$ is at most 2).

Set $R:= V(G)\setminus (A_y^2\cup A_{12})$. Note that $R$ can be partitioned into two red subgraphs of diameter at most 2: $K_x$ and a red star from $y$.
If every vertex in $A_y^2$ sends a red edge to $R$, then using Proposition \ref{szor} (ii) we have a cover by two red components of diameter at most 4 (note that
$A_{12}$ can be added to the first one containing $K_x$).

Thus we may assume that there is a vertex $z\in A_y^2$ such that every edge from $z$ to $R$ is blue.
Let $Z= \{ v \; | \; v\in R, (z,v)\not\in E(G) \}$. Then $Z$ spans a complete graph in $G$, so by Proposition \ref{simplefolk} it has diameter at most 3 in one of the colors.
This is one component of the cover while the other one is
$$G_2[\{y\}\cup A_y^2 \cup A_{12} \cup N_2(z)],$$
having also diameter at most 3 (by Proposition \ref{szor}(i)), since it is a star from $y$ with another star added to it from one of the leaves, namely from $z$. This finishes the proof. \qed

\subsection{Proof of Theorem \ref{antiby3}}
\begin{proof} Let $v_1\in V(K'),v_2\in V(K'')$ denote the vertices not adjacent to $v$.
Consider a 2-coloring of $G$ with colors 1 and 2. Applying Lemma \ref{folk} in $K'$ and $K''$, we have two cases.

{\bf Case 1:} $K'$ and $K''$ both have diameter at most $2$ in one of the colors, say in colors $c_1,c_2$, respectively. Let $\bar{c_i}$ denote the color that is different from $c_i$.
We claim that there is a $(3,2)$-cover of $G$. Indeed, if any edge in $[v,V(K')]$ is of color $c_1$ then $V(K')\cup \{v\}$ is of diameter at most $3$ in color $c_1$ and we have a $(3,2)$-cover.
Similarly, if any edge in $[v,V(K'')]$ is of color $c_2$ then $V(K'')\cup \{v\}$ is of diameter at most $3$ in color $c_2$ and we have a $(3,2)$-cover again.
Otherwise all edges of $[v,V(K')]$ are of color $\bar{c_1}$ and all edges of $[v,V(K'')]$ are of color $\bar{c_2}$. If $v_1$ sends a color $\bar{c_1}$ edge to $V(K')$ or $v_2$ sends a color $\bar{c_2}$ edge to $V(K'')$ then again, we have a $(3,2)$-cover. Otherwise all edges of $[v_1,V(K')]$ are of color $c_1$ and all edges of $[v_2,V(K'')]$ are of color $c_2$.
Finally, consider the edge $(v_1,v_2)$. If $c_1=c_2$, then the star from $v$ and the $(v_1,v_2)$ edge form a (2,1)-cover. If $c_1\not = c_2$, wlog let the color of the $(v_1,v_2)$ edge be $c_1$ (symmetric in the other case).
We can add the edge $(v_1,v_2)$ to the $c_1$-colored star from $v_1$ and together with the $c_1$-colored star from $v$ we get a (2,2)-cover.

{\bf Case 2:} At least one of $K',K''$, say $K'$ has diameter 3 in both colors. By (ii) in Lemma \ref{folk}, $K'$ has monochromatic spanning double stars $T_1$ in color 1, $T_2$ in color 2, where the base edges, $e=(x_1,x_2)$ in color 1 and $f=(x_3,x_4)$ in color 2, span a $P_4^2$. One of the two base edges, say $e$, is not incident to $v_1$. Thus $g=(v,x_1),h=(v,x_2)$ are both edges of $G$. If $g$ or $h$ has color 1 then $T_1\cup \{v\}$ is a double star in color 1 and together with $K''$ we get a $(3,3)$-cover of $G$. Otherwise  $g,h$ both have color 2 and we get a five cycle in color 2 on vertices $\{v,x_1,x_2,x_3,x_4\}$. Now $T_2\cup \{v\}$ is still of diameter $3$ in color 2, because from $v$ we have a path of length 2 in color 2 to {\em both endpoints} of the base edge of $T_2$.
\end{proof}

\subsection{Two proofs of Theorem \ref{3/2}}

We proceed by induction on $\alpha$. The statement is clearly true for $\alpha=1$ (complete graph), by Proposition \ref{simplefolk} we can cover by one
monochromatic connected component of diameter at most 3. It is also true for $\alpha =2$ by Theorem \ref{main} (actually 2 components are sufficient instead of 3).
Take a general $\alpha>2$ and assume that the statement is true
for every $1\leq \alpha' < \alpha$. Let $G=(V,E)$ be a graph with $\alpha(G)=\alpha$.

Applying a deep tool, one can give a short proof as follows. If $G$ is a perfect graph, then Proposition \ref{perfectgraphs} gives a much stronger result than we are looking for. Otherwise, by the strong perfect graph theorem \cite{CRST}, $G$ must contain an odd hole $C$ or an odd antihole $A$. In the first case we must have two consecutive edges on $C$ with the same color, in the second case the odd cycle formed by the long diagonals of $A$ must have two consecutive edges of the same color.
Thus in both cases there are two non-adjacent vertices $x$ and $y$, which are connected by a monochromatic (say red) path of length 2.

Now we can cover $\{x,y\}\cup N(x)\cup N(y)$ by three monochromatic connected components of diameter at most 4.
Indeed, the two red stars from $x$ and $y$ connect to one red component of diameter at most 4 by the assumption.
Let $R$ denote the leftover vertices.
The vertices of $R$ are not adjacent to $x$ and $y$, thus we have $\alpha (G|_{R}) \leq \alpha -2$, since otherwise we can add $x$ and $y$ to get an independent set of size at least $\alpha + 1$ in $G$, a contradiction.
Hence, by induction $G|_{R}$ can be covered by $\lfloor 3(\alpha - 2)/2\rfloor = \lfloor 3\alpha/2 \rfloor - 3$ monochromatic connected components of diameter at most 4. Then altogether we have
$\lfloor 3\alpha/2 \rfloor - 3 +3 = \lfloor 3\alpha/2 \rfloor$ monochromatic connected components of diameter at most 4 covering $G$, as desired.

\medskip

To avoid using \cite{CRST}, one can finish the proof as follows. We {\em may assume} that for every non-adjacent pair of vertices in $G$ there is no monochromatic path of length 2 connecting them because otherwise as above we can finish by induction. Consider an independent set $I= \{ u_1, u_2, \ldots , u_{\alpha} \}$ of size $\alpha$ in $G$.

\begin{claim}
Every vertex in $v\in (V\setminus I)$ sends one or two edges to $I$. Furthermore, if $v$ sends two edges, they must have a different color.
\end{claim}

Indeed, it is not possible that $v$ sends no edges to $I$, because then we could add $v$ to $I$ to get an independent set of size $\alpha + 1$ in $G$, a contradiction.
Furthermore, $v$ cannot send more than 2 edges by the assumption, because then we would have a pair of non-adjacent vertices with a monochromatic path of length 2 connecting them, a contradiction.
For the same reason, if $v$ sends two edges to $I$, they must have a different color.

Let us label the vertices in $V\setminus I$ by their one or two neighbors in $I$. Let $V_1$ denote the label-1, $V_2$ the label-2 vertices.
Denote the set of vertices in $V_1$ which share the same label $u_i\in I$ by $V_1^i$.
\begin{claim}
The set $u_i \cup V_1^i$ spans a complete graph in $G$.
\end{claim}
Indeed, if an edge is missing in $V_1^i$, then we can add the two vertices to $I\setminus \{u_i\}$ to get an independent set of size $\alpha + 1$ in $G$, a contradiction.
Thus using Proposition \ref{simplefolk} for each $1\leq i \leq \alpha$ we can cover $u_i \cup V_1^i$ by one
monochromatic connected component of diameter at most 3. Label $u_i$ with red or blue according to the color of this component
covering $u_i \cup V_1^i$. Let $I_1$ denote the set of those vertices in $I$ that are labeled with red, and $I_2$ that are labeled with blue.
Assume wlog that $|I_2|\leq |I_1|$, then $|I_2| \leq \lfloor \alpha/2 \rfloor$.

Next consider $V_2$. If a vertex $v\in V_2$ sends a red edge to a $u_i\in I_1$ (similarly for blue), we can add $v$ to the
red component covering $u_i \cup V_1^i$. By Proposition \ref{szor} the diameter is still at most 4.
Thus the only vertices left in $V_2$ are precisely those that send a red edge to $I_2$ and a blue edge to $I_1$.
These vertices can be covered by $|I_2|$ red stars from the vertices in $I_2$. Thus altogether we have
$\alpha + \lfloor \alpha/2 \rfloor = \lfloor 3\alpha/2 \rfloor$ monochromatic connected components of diameter at most 4 covering $G$, as desired. \qed

\section{Conclusion and some open problems}
\begin{itemize}
\item  We have shown that every 2-colored graph $G$ with $\alpha(G)=2$ has a $(4,4)$-cover. This is not true for $(2,2)$-covers by Remark \ref{const}, thus the obvious open question left is whether there is a $(3,4)$-cover or even a $(3,3)$-cover?

\item In a similar vein, can Theorem \ref{antiby3} be strengthened by stating that there is always a $(3,2)$-cover?

\item The following is an interesting conjecture of English, McCourt, Mattes, and  Phillips \cite{EMMP}. Assume that $n$ is even and let $G$ be a graph obtained from $K_n$ by the removal of the edges of a perfect matching (thus $\alpha(G)=2$). Then $V(G)$ has a $(2,2)$-cover.

\item The following question is related to Subsection \ref{orth}. Let $r=2$, $d=3$. It is possible that $g(2,3,\alpha)=\alpha$ and that would give the best possible answer to the problem from \cite{DGHS} discussed in the introduction: $f(2,\alpha)=3$.  However, presently the best we know is $\alpha\le g(2,3,\alpha)\le 2\alpha-1$. The first step would be to decide whether $g(2,3,2)$ is 2 or 3.
\end{itemize}

\end{document}